\documentclass[11pt,reqno]{amsart}
  \usepackage{latexsym} 
  \usepackage{mathtools}
  \usepackage[all]{xy}
  \usepackage{amsfonts} 
  \usepackage{amsthm} 
  \usepackage{amsmath} 
  \usepackage{amssymb}
  \usepackage{pifont}  
  \usepackage{enumerate}
  \usepackage{dcolumn}
\usepackage{tikz}
\usetikzlibrary{arrows, calc, decorations.markings, positioning,shapes.geometric}
\usepackage{tikz-cd}
  \usepackage{comment}
  \usepackage{hyperref}  
\newcolumntype{2}{D{.}{}{2.0}}
  \xyoption{2cell}

   \def\<{{\langle}} 
  \def\>{{\rangle}}

  \def\note#1{{}}

  \def\note#1{}

  \def\beq{\begin{equation}} 
  \def\eeq{\end{equation}}

  \newcounter{zlist}

  \newcounter{blist}

  \newcounter{rlist}

\def\stac#1{\raise-.2cm\hbox{$\stackrel{\displaystyle\otimes}{\scriptscriptstyle{#1}}$}}

\def\cten#1{\raise-.2cm\hbox{$\stackrel{\displaystyle\reallywidehat{\otimes}}
{\scriptscriptstyle{#1}}$}}

  \def\Label#1{\label{#1}\ifmmode\llap{[#1] }\else 
  \marginpar{\smash{\hbox{\tiny [#1]}}}\fi} 
  \def\Label{\label}

  \newtheorem{proposition}{Proposition}[section]
  \newtheorem{lemma}[proposition]{Lemma} 
  \newtheorem{corollary}[proposition]{Corollary} 
  \newtheorem{theorem}[proposition]{Theorem} 
  
\theoremstyle{definition} 
  \newtheorem{definition}[proposition]{Definition}
    
  \newtheorem{example}[proposition]{Example}

  \theoremstyle{remark} 
  \newtheorem{remark}[proposition]{Remark}

  \newcounter{c} 
   
  \newcommand{\etyk}[1]{\vspace{-7.4mm}$$\begin{equation}\Label{#1} 
  \addtocounter{c}{1}} 
  \renewcommand{\]}{\ifnum \value{c}=1 $$\else \end{equation}\fi} 
   \numberwithin{equation}{section}

\newcommand{\LeftB}[1]{\,\stackon[-1.6ex]{\scalebox{.8}{\,\,$#1$}}{\scalebox{2.2}{$\triangleleft$}}\,}
\newcommand{\MiddleB}[1]{\,\stackon[-1.3ex]{\scalebox{.8}{$#1$}}{\scalebox{1.2}{$\square$}}\,}
\newcommand{\RightB}[1]{\,\stackon[-1.6ex]{\scalebox{.8}{$#1$\,\,}}{\scalebox{2.2}{$\triangleright$}}\,}

\def\*C{{}^*\hspace*{-1pt}{\Cc}}

\def\text#1{{\rm {\rm #1}}}

\usepackage{scalerel,stackengine}
\stackMath
\newcommand\reallywidehat[1]{%
\savestack{\tmpbox}{\stretchto{%
  \scaleto{%
    \scalerel*[\widthof{\ensuremath{#1}}]{\kern.1pt\mathchar"0362\kern.1pt}%
    {\rule{0ex}{\textheight}}
  }{\textheight}%
}{2.4ex}}%
\stackon[-6.9pt]{#1}{\tmpbox}%
}
\usepackage[
backend=biber,
style=alphabetic,
sorting=ynt
]{biblatex}
\begin{document}

\title{Biunit pairs in semiheaps and associated semigroups}

\begin{abstract}
Biunit pairs are introduced as pairs of elements in a semiheap that generalize the notion of unit. Families of functions generalizing involutions and conjugations, called switches and warps, are investigated. The main theorem establishes that there is a one-to-one correspondence between monoids equipped with a particular switch and semiheaps with a biunit pair. This generalizes a well-established result in semiheap theory that connects involuted semigroups and semiheaps with biunit elements. Furthermore, diheaps are introduced as semiheaps whose elements belong to biunit pairs and they are shown to be non-isomorphic but warp-equivalent to heaps.
\end{abstract}

\author{Bernard  Rybo{\l}owicz}
\address{(B. Rybo\l owicz) Department of Mathematics, Heriot-Watt University, Edinburgh EH14 4AS, and Maxwell Institute for Mathematical Sciences, Edinburgh, UK}
\email{B.Rybolowicz@hw.ac.uk}
\author{Carlos Zapata-Carratal\'a}
\address{(Carlos Zapata-Carratal\'a)
Edinburgh Mathematical Physics Group, United Kingdom}

\email{C.Zapata.Carratala@gmail.com}


\baselineskip=15pt
\date\today
\maketitle

\section*{Introduction}

The study of semiheaps can be traced back to the early work on semigroup theory by R.\ Baer \cite{Baer1929heap}, H.\ Pr\"ufer \cite{Prufer1924heap} and A.K.\ Su\v skevi\v c \cite{Susk1937heap}. They investigated what were later known as heaps, certain kinds of ternary algebras arising from groups and semigroups satisfying some special conditions, see definitions \ref{def:semiheap} and \ref{def:heap}. Heaps and groups are closely related. From a heap $(H,[-,-,-])$, we can construct a group $\mathrm{G}(H)$ and vice-versa. That is achieved by fixing an element $a\in H$ and currying it in the middle argument of the ternary operation $[-,-,-]$ to acquire a binary one $(x,y)\xmapsto{[-,a,-]}[x,a,y]$. Conversely, we associate a heap $\mathrm{H}(G)$ with a group $(G,\cdot)$ by taking the ternary bracket $[a,b,c]:=a\cdot b^{-1}\cdot c$, where $a,b,c\in G$. If we start with a ternary operation, these constructions compose to the original structure. In the other direction, starting with a group, we only obtain the group up to isomorphism as we forget which element is fixed, but we always find a so-called variant, see \cite{jbhickney1983}. In 1953, V. V. Wagner \cite{wagner1953theory} introduced semiheaps in the context of partial functions and relations on sets as a way to formalize the theory of coordinate charts on manifolds. A \textbf{semiheap} is a straightforward generalisation of a heap: a ternary algebra with the generalized associativity property satisfied by heaps but without any cancelling behaviour of elements, that is $(S,[-,-,-])$ such that
\begin{equation*}
    [[a,b,c],d,e]=[a,[d,c,b],e]=[a,b,[c,d,e]]
\end{equation*}
for all $a,b,c,d,e\in S$. These ternary algebras have applications to Morita equivalence \cite{lawson2011generalized}, pseudogroups\cite{kock2007principal}, affine structures \cite{hawthorn2011near} \cite{breaz2022heaps}, Lie theory and quantum mechanics \cite{brzezinski2022trusses} \cite{bruce2022semiheaps}, and hypermatrix theory \cite{abramov2009algebras} \cite{zapata2022heaps}. In contrast to groups and heaps, there is no correspondence between semiheaps and semigroups. We can associate with every semiheap a semigroup by fixing an element. For different choices, however, semigroups do not have to be isomorphic anymore. The opposite direction is not unique either, more data is required. In his seminal paper \cite{wagner1953theory}, Wagner shows that there is a correspondence between involuted monoids and semiheaps with \textbf{biunit elements}, that is elements $e\in S$ satisfying
\begin{equation*}
    [e,e,x]=x=[x,e,e]
\end{equation*}
for all $a\in S$. If there exists a biunit element, one can build an involuted monoid from a semiheap, and then reconstruct it from the monoid. Similarly, for monoids with involutive anti-homomorphism, we  construct a semiheap with biunit from an involuted monoid, and then reconstruct the monoid up to isomorphism. Our contribution is to consider a more parsimonious generalization of biunit elements in semiheaps by considering a pair $(a,b)$ such that:
\begin{equation*}
    [a,b,x]=x=[x,a,b]
\end{equation*}
which we call a \textbf{biunit pair}. The main purpose of our paper is to extend the classical result of Wagner \cite{wagner1953theory} to semiheaps with biunit pairs and determine the appropriate class of semigroups that establishes the correspondence.

We begin in Section \ref{sec:pre} by recalling some generalities about ternary operations and semiheaps. We establish our notation and conventions by introducing ternars, currying and semiheaps.

In Section \ref{sec:t-w-s}, we introduce twistings of semiheaps, i.e. a pair $S_\varphi:=(S,[-,\varphi(-),-])$ is a $\varphi$-twist of $(S,[-,-,-])$ if $S_\varphi$ is a semiheap, and function $\varphi:S \to S$ is called a twist. In Definition \ref{def:warp}, we introduce a special case of a twist called a warp. We show that in the case of abelian semiheap or existence of particular injective currying, twists coincide with warps, see Lemma \ref{lem:injsh}. We observe that the relation of differing by a bijective warp gives an equivalence relation on the set of all twistings of a semiheap. Similarly, we introduce an function $\psi:S\to S$ called a switch. A switch allows us to associate with a semigroup $(S,\cdot)$, a semiheap $(S,[-,-,-]^{\psi})$, where $[a,b,c]^\psi:=a\cdot\psi(b)\cdot c$ for all $a,b,c\in S$. If a semigroup $S$ has at least one left and one right cancellative element, switches correspond to all semiheap structures given by the bracket $[-,-,-]^\psi$, see Lemma \ref{lem:equiv-sg-sh}. In the case of abelian semiheaps, switches coincide with warps of $(S,[-,-,-]^\mathrm{Id})$ see Lemma \ref{lem:cor-switch-warp}. Proposition \ref{prop:sg:list} gives us the properties of switches for monoids. The proposition is crucial to prove the main theorem of the paper. The most important property is that if $\varphi$ is a switch for some monoid $S$ and there exist $u\in S$ such that $\varphi(u)=1$, then $u$ is invertible and $u^{-1}=\varphi(1)$.

Section \ref{sec:main} introduces biunit pairs a pairs in a semiheap $S$ in Definition \ref{def:bp}. A biunit pair is a generalisation of biunits, as any biunit element $a\in S$ (see Definition \ref{def:biunit}) belongs to a biunit pair $(a,a)$ in $S$. For curryings, lower arity functions created by fixing elements in a semiheap ternary operation, given by biunit pairs, left and right shifts give us bijective warps. That said, semiheaps $S$, $S_{\lambda_{aa}}$, $S_{\lambda_{bb}}$, $S_{\rho_{aa}}$, and $S_{\rho_{bb}}$ are warp equivalent see Corollary \ref{cor:warpequiv}. We observe that if $(S,[-,-,-])$ is a heap with a biunit pair $(a,b)$, then the currying $ c\xmapsto{\mu_{bb}} [b,c,b]$ is a switch for a monoid $(S,\MiddleB{a},b)$, where $x\MiddleB{a}y:=[x,a,y]$ for all $x,y\in S$, $b$ is an identity, and $\mu_{bb}(a)=b$, see Lemma \ref{lem:mainproof1} and corollary \ref{cor:mainproof2}. These two observations with Proposition \ref{prop:sg:list} allows us to formulate and prove our main result in Theorem \ref{thm:main}. The theorem states that there is a one-to-one correspondence between monoids with an invertible switch and semiheaps with a choice of biunit pair. We conclude our contribution by motivating the definition of \textbf{diheaps} as semiheaps whose elements belong to biunit pairs and showing that these previously unknown structures are related but are not isomorphic to heaps. Proposition \ref{prop:equiv} shows that diheaps and heaps are warp-equivalent.

\section{Preliminaries}\label{sec:pre}

Throughout this work we will consider a set $T$ together with a ternary operation, that is a map $[-,-,-]_T:T\times T\times T\to T$. We will call a pair $(T,[-,-,-]_T)$ a {\bf ternar}. Following the universal algebra approach a morphism of ternars $(T,[-,-,-]_T)$ and $(H,[-,-,-]_H)$ is a function $\varphi:T\to H$ such that $\varphi([a,b,c]_T)=[\varphi(a),\varphi(b),\varphi(c)]_H$, for all $a,b,c\in T$, see \cite{bergman2015invitation}. We shall omit index of the ternary operation and write $(T,[-,-,-]):=(T,[-,-,-]_T)$, keeping in mind on which set bracket $[-,-,-]$ is defined.

Given a ternar $(T,[-,-,-])$, lower arity binary operations arise in a natural way by currying elements into the ternary operation in different ways. By currying a single element $a$ we get the following three binary operations:
\begin{equation}\label{def:binaries}
    x \LeftB{a} y := [a,x,y] \qquad x \MiddleB{a} y := [x,a,y] \qquad x \RightB{a} y := [x,y,a] 
\end{equation}
for $x,y\in T$. For our purposes the most important will be the middle one: the pair $(T,\MiddleB{a})$ will be called \textbf{$a$-retract} of the ternar $T$. Similarly, by currying a pair of elements $(a,b) \in T\times T$ we obtain the following three unary operations:
\begin{equation}\label{def:unaries}
    \lambda_{ab}(x):=[a,b,x] \qquad \mu_{ab}(x):= [a,x,b] \qquad \rho_{ab}:=[x,a,b]
\end{equation}
for $x\in T$. These are endofunctions of the underlying set $T$, so we call them, respectively, the \textbf{left}, \textbf{middle} and \textbf{right} \textbf{shifts} of the ternar $(T,[-,-,-])$. 

In this paper we will mainly consider a particular class of ternars called semiheaps.

\begin{definition}\label{def:semiheap}
A {\bf semiheap} is a set $S$ together with a ternary operation $[-,-,-]:S\times S\times S\to S$ such that for all $a,b,c,d,e\in S$,
\begin{equation}\label{semiheap}
    [[a,b,c],d,e]=[a,[d,c,b],e]=[a,b,[c,d,e]]. 
\end{equation}
\end{definition}

\noindent Note that the middle term in (\ref{semiheap}) implies a departure from (ordinary) sequential associativity i.e. the property that notationally allows for brackets to be dropped. Consequently, property (\ref{semiheap}) has been called weak associativity \cite{wagner1951ternary}, associativity of the second kind \cite{carlsson1976cohomology}, quasi-associativity \cite{kolar2000heap}, para-associativity \cite{hawthorn2009radical}, type B associativity \cite{kerner2008ternary} and pseudo-associativity \cite{hollings2017wagner}.

\begin{definition}
Given a semiheap $(S,\overline{[-,-,-]})$, a simple check shows that if the ternary operation $[-,-,-]$ satisfies (\ref{semiheap}) then so does $\overline{[a,b,c]}:=[c,b,a]$; $(S,\overline{[-,-,-]})$ is called the \textbf{reverse semiheap}. A semiheap $(S,[-,-,-])$ is called \textbf{abelian} when $[-,-,-]=\overline{[-,-,-]}$.
\end{definition}

\begin{example}
The classical example of a semiheap is the compositional structure of the set of binary relations between two arbitrary sets $\textsf{Rel}(A,B)$ given by the following ternary operation:
\begin{equation*}
    [R_1,R_2,R_3]_\textsf{Rel} := R_1 \circ R_2^\top \circ R_3
\end{equation*}
where $R_1,R_2,R_3\subset A\times B$ are binary relations, $\circ$ is the usual composition of relations and $\top$ denotes the transpose or converse relation. It is easy to see that (\ref{semiheap}) holds as a direct consequence of the associativity property of $\circ$ and the fact that $\top$ is an involution and a $\circ$-antihomomorphism. More generally, the homsets of any dagger category $(\mathcal{C},\dag)$ carry a semiheap structure via the analogous construction:
\begin{equation*}
    [f,g,h]_\mathcal{C} := f \circ g^\dag \circ h
\end{equation*}
for $f,g,h\in \mathcal{C}(A,B)$ and any two objects $A,B\in \mathcal{C}$. The commutative diagram that realizes this ternary operation in a dagger category displays a characteristic zigzag pattern:
\begin{equation}\label{dag}
\begin{tikzcd}[sep=large]
A \arrow[rr, "f", bend left] \arrow[rr, "h", bend right] & & B \arrow[ll, "g^\dag"']
\end{tikzcd} 
\end{equation}
\end{example}
\begin{example}\label{exm:sg-to-sh}
Let $G$ be a set together with an associative binary operation $\cdot$ and a function $*:G\to G$ satisfying $(a^*)^*=a$ and $(a\cdot b)^*=b^*\cdot a^*$ for all $a,b\in G$, that is, $*$ is an antihomomorphic involution, then it is easy to show that the ternary operation defined by
\begin{equation*}
    [a,b,c]_G := a \cdot b^*\cdot c
\end{equation*}
satisfies (\ref{semiheap}) thus making $(G,[-,-,-]_G)$ into a semiheap. In particular, all groups carry a canonical semiheap structure via the above construction by considering the inversion involution $*=\,^{-1}$. 
\end{example}

Note that all these examples correspond to ternary operations constructed from an associative binary operation that is \emph{twisted} by some map, in reference to the diagram (\ref{dag}).

\begin{definition}\label{def:biunit}
Let $(S,[-,-,-])$ be a semiheap, we say that element $e\in S$  is a {\bf left (right) Mal'cev element} or a {\bf left (right) biunit} if for all $s\in S$
$$
[e,e,s]=s \qquad ([s,e,e]=s).
$$
If an element $e\in S$ is a left and right biunit, we call it a {\bf biunit} or {\bf Mal'cev element}.
\end{definition}

The existence of Mal'cev elements in a semiheap allows us to construct an involuted monoid.

\begin{theorem}[Involuted Monoid of a Semiheap with Biunit, Thm. 8.2.8-9 \cite{hollings2017wagner} or Thm. 2.8-9 \cite{wagner1953theory}] \label{monoidsemiheap}
Let $(S,[-,-,-])$ be a semiheap and $e\in S$ be a Mal'cev element, then the binary operation
\begin{equation*}
    a\MiddleB{e} b := [a,e,b]
\end{equation*}
and the map
\begin{equation*}
    \mu_{ee}(a) := [e,a,e]
\end{equation*}
make $(S,\MiddleB{e},\mu_{ee})$ into an involuted monoid with $e$ as identity element such that the original semiheap structure is recovered by the ternary operation:
\begin{equation*}
    [a,b,c]=a \MiddleB{e} \mu_{ee}(b) \MiddleB{e} c.
\end{equation*}
Furthermore, if there exists another biunit $u\in S$, then there is an isomorphism of involuted monoids $\rho_{eu}:(S,\MiddleB{e},\mu_{ee})\to (S,\MiddleB{u},\mu_{uu})$ given by
\begin{equation*}
    \rho_{eu}(a) := [a,e,u].
\end{equation*}
\end{theorem}

\begin{corollary}
There is a one-to-one correspondence between semiheaps with a chosen biunit element and involuted monoids:
\begin{equation*}
    \{ \text{semiheap} + \text{biunit} \,\, \text{element} \} \rightleftharpoons \{ \text{involuted} \,\, \text{monoid} \}
\end{equation*}
which is realized explicitly via the assignments:
$$
(S,e)\mapsto (S,\MiddleB{e},\mu_{ee})\quad \text{and} \quad (M,\cdot,1,\varphi)\mapsto ((M,[a,b,c]:=a\cdot\varphi(b)\cdot c), 1).
$$
\end{corollary}
 The connection between semiheaps and involuted monoids is, in fact, more general.
\begin{theorem}[Semiheaps Embed into Involution Monoids, Thm. 8.2.10-11 \cite{hollings2017wagner} or Thm. 2.10-11 \cite{wagner1953theory}] \label{embedsemiheap}
Any semiheap can be homomorphically embedded into a semiheap with a Mal'cev element. Therefore, it follows from Theorem \ref{monoidsemiheap} that any semiheap can be embedded into an involuted monoid.
\end{theorem}

\begin{remark}
In the case of a semiheap $(S,[-,-,-])$ with a biunit $e\in S$, we get that $a$-retract of $S$ is a variant $(S,[e,a,e])$ of $e$-retract of $S$, see \cite{jbhickney1983}.
\end{remark}

\begin{definition}\label{def:heap}
A {\bf heap} is a semiheap $(H,[-,-,-])$ in which every element is a Mal'cev element.
\end{definition}

\begin{example}
If $(G,\cdot)$ is a group, then a pair $(G,[-,-,-])$, where $[a,b,c]=a\cdot b^{-1}\cdot c$ for all $a,b,c\in G$, is a heap.
\end{example}

Note that Theorems \ref{monoidsemiheap} and \ref{embedsemiheap} restrict to equivalent statements for heaps and groups.

\section{Twists, Warps and Switches}\label{sec:t-w-s}

In this section, we study the twistings of semiheaps. By that, we understand a method to deform a ternary operation to acquire a new semiheap. Furthermore, we research connections among semigroups, semiheaps and twisted semiheaps. Let us start by clarifying what we understand by twists.

\begin{definition}\label{def:twist}
Let $(S,[-,-,-])$ be a semiheap and $\varphi:S\to S$ be a function. We say that that $\varphi$ is a \textbf{twist} if $(S,[-,-,-]_\varphi)$, where $[a,b,c]_\varphi:=[a,\varphi(b),c]$ for all $a,b,c\in S$, is a semiheap. The semiheap $S_\varphi:=(S,[-,-,-]_\varphi)$ is called the {\bf $\varphi$-twist of the semiheap $(S,[-,-,-])$}.
\end{definition}

We will be particularly interested in semiheaps twisted by warps.

\begin{definition}\label{def:warp}
Let $(S,[-,-,-])$ be a semiheap. We say that a function $\eta:S\to S$ is a {\bf warp}  if for all $a,b,c\in S$,
$$
\eta([a,\eta(b),c])=[\eta(a),b,\eta(c)].
$$
\end{definition}

\begin{lemma}\label{lem:shifts}
Let $(S,[-,-,-])$ be a semiheap. Then for any $e\in S$, the shifts $\rho_{ee}$ and $\lambda_{ee}$ defined in \eqref{def:unaries} are warps.
\end{lemma}
\begin{proof}
Observe that for all $a,b,c\in S$,
$$
\rho_{ee}([a,\rho_{ee}(b),c])=[[a,[b,e,e],c],e,e]=[[a,e,e],b,[c,e,e]]=[\rho_{ee}(a),b,\rho_{ee}(c)].
$$
Thus $\rho_{ee}$ is a warp. Similarly one can show that $\lambda_{ee}$ is a warp.
\end{proof}

The following lemmata establish the connection between twists and warps.

\begin{lemma}\label{lem:sgtosh}
Let $(S,[-,-,-])$ be a semiheap and $\eta:S\to S$ be a warp. Then $S_\eta$ is a semiheap. In other words, warps are semiheap twists.
\end{lemma}
\begin{proof}
Simple check that for all $a,b,c,d,e\in S$,
$$
[[a,b,c]_\eta, d,e]_\eta=[[a,\eta(b),c],\eta(d),e]=[a,\eta(b),[c,\eta(d),e]]=[a,b,[c,d,e]_\eta]_\eta,
$$
where the second equality follows by the associativity of $[-,-,-]$. Again, using associativity of $[-,-,-]$, we get that 
$$
\begin{aligned}
[[a,b,c]_\eta,d,e]_\eta &=[[a,\eta(b),c],\eta(d),e]=[a,[\eta(d),c,\eta(b)],e]=[a,\eta([d,\eta(c),b]),e]\\ &=[a,\eta([d,c,b]_\eta),e]=[a,[d,c,b]_\eta, e]_\eta.
\end{aligned}
$$
Thus, $(S,[-,-,-]_\eta)$ is a semiheap.
\end{proof}
\noindent Although twists are more general than warps, under some conditions on the currying $\mu$, see \eqref{def:unaries}, we can obtain a sufficient condition for warps and twists to coincide in a given semiheap.
\begin{lemma}\label{lem:injsh}
Let $(S,[-,-,-])$ be a semiheap and assume $a,b\in S$ are such that $\mu_{ab}:S\to S$ is injective. Then $S_\varphi$ is a semiheap if and only if $\varphi$ is a warp.
\end{lemma}
\begin{proof}
If $\varphi$ is a warp then by Lemma \ref{lem:sgtosh} $(S,[-,-,-]_\varphi)$ is a semiheap.
The opposite way, by the associativity of  semiheaps $(S,[-,-,-]_\varphi)$ and $(S,[-,-,-])$, we get that for all $c,d,e\in S$,
$$
\begin{aligned}
[a,\varphi([c,\varphi(d),e]),b]&=[a,[c,d,e]_\varphi,b]_{\varphi}=[[a,e,d]_\varphi,c,b]_\varphi=[[a,\varphi(e),d],\varphi(c),b]\\ &=[a,[\varphi(c),d,\varphi(e)],b],
\end{aligned}
$$
where the third equality follows by associativity of $[-,-,-]_\varphi$ and the fifth by the associativity of $[-,-,-]$. Now, since $\mu_{ab}$ is injective we get that
$$\varphi([c,\varphi(d),e])=[\varphi(c),d,\varphi(e)],$$
hence $\varphi$ is a warp.
\end{proof}

The following lemmata provides us with a class of warps in abelian semiheaps.

\begin{lemma}\label{lem:mu:absem}
Let $(S,[-,-,-])$ be an abelian semiheap. Then for any element $e\in S$, a map \mbox{$b\xmapsto{\mu_{ee}}[e,b,e]$} is a warp.
\end{lemma}
\begin{proof}
Simple observation that,
$$
\begin{aligned}
\mu_{ee}([a,\mu_{ee}(b),c])&=[e,[a,[e,b,e],c],e]=[e,[[a,e,b],e,c],e]=[[e,c,e],[a,e,b],e]\\ &=[[e,c,e],b,[e,a,e]]=[\mu_{ee}(c),b,\mu_{ee}(a)],
\end{aligned}
$$
where the third equality follows by the associativity of $S$. Since $S$ is abelian, we get that $$[\mu_{ee}(c),b,\mu_{ee}(a)]=[\mu_{ee}(a),b,\mu_{ee}(c)],$$ and $\mu_{ee}$ is a warp
\end{proof}

\begin{lemma}
Let $(S,[-,-,-])$ be a semiheap and $\mu_{ee}$ be surjective. Then $\mu_{ee}$ is a warp  if and only if $S$ is abelian semiheap.
\end{lemma}
\begin{proof}
Observe, that if $\mu_{ee}$ is a warp, then for all $h,g,b\in S$ exist $a,c\in S$ such that $h=\mu_{ee}(a)$ and $g=\mu_{ee}(c)$. The property of warp gives $$[[e,a,e],b,[e,c,e]=[[e,c,e],b,[e,a,e]],$$
and by putting $h$ and $g$, we get
$$
[h,b,g]=[g,b,h],
$$
so $S$ is abelian semiheap.

The opposite way follows by Lemma \ref{lem:mu:absem}.
\end{proof}

\begin{lemma}
Let $(S,[-,-,-])$ be an abelian semiheap. If $\varphi$ is a warp such that $\varphi$ is an involution, that is $\varphi^2=\mathrm{Id}$, and there exists $e\in S$ such that $\varphi(e)=e$, then $\mu_{ee}\circ \varphi$ is also a warp.
\end{lemma}
\begin{proof}
Let $a,b,c\in S$, then

$$
\begin{aligned}
\mu_{ee}\circ \varphi([a,\mu_{ee}\circ\varphi(b),c])&=[e,\varphi([a,[e,\varphi(b),e],c]),e]=[e,\varphi([a,[\varphi(e),\varphi(b),\varphi(e)],c]),e]\\ &=[e,\varphi([a,\varphi([e,\varphi^2(b),e]),c]),e]=[e,\varphi([a,\varphi([e,b,e]),c]),e]\\ &=[e,[\varphi(a),[e,b,e],\varphi(c)],e]=[\mu_{ee}\circ\varphi(a),b,\mu_{ee}\circ\varphi(c)],
\end{aligned}
$$
where the third equality uses the assumption that $\varphi(e)=e$, the fourth uses involutivity, the fifth follows by the fact that $\varphi$ is a warp and the sixth is associativity and abelianity.
Thus, $\mu_{ee}\circ \varphi$ is a warp.
\end{proof}

\begin{remark}
It is worth noting that in general, a composition of warps is not a warp. Let $S$ be a semiheap, $\varphi$ be a warp on $S$ and $\psi$ be a warp on $S_\varphi$. Then $\psi\circ\varphi$ is not necessarily a warp on $S$. One can easily show that if $\psi\circ \varphi=\varphi\circ\psi$, then $\psi\circ\varphi$ is a warp on $S$.
\end{remark}

\begin{lemma}\label{lem:invwrap}
Let $(S,[-,-,-])$ be a semiheap and $\eta$ be a bijection and a warp. Then $\eta^{-1}$ is also a warp.
\end{lemma}

\begin{proof}
Let $a,b,c\in S$ and observe that since $\eta$ is a bijection there exist $e,d,f$ such that $\eta(e)=a$, $\eta(b)=d$ and $\eta(c)=f$.Then
$$\eta^{-1}([a,\eta^{-1}(b),c])=\eta^{-1}([\eta(e),\eta^{-1}(\eta(d)),\eta(f)])=\eta^{-1}(\eta([e,\eta(d),f]))=[\eta^{-1}(a),b,\eta^{-1}(c)].
$$
Thus $\eta^{-1}$ is a warp.
\end{proof}

Observe that since the inverse of the bijective warp is a warp, we get a criterion for the equivalence of semiheap structures on the same set.

\begin{definition}
We say that two semiheaps $S$ and $H$ are {\bf warp equivalent} if there exist a sequence of bijective warps $(\varphi_i)_{i\in\{0,1,2,\ldots,n\}}$, for $n\in \mathbb{N}$, such that $H=(((S_{\varphi_1})_{\varphi_2})...)_{\varphi_n}$, where $\varphi_i$ is a warp of $S_{\varphi_{i-1}}$ and $\varphi_0=\mathrm{Id}$.
\end{definition}

Indeed a relation $S\sim H$ if and only if $S$ is warp equivalent to $H$ is an equivalence relation on the set of all twists of $S$. Indeed $S\sim S$ as an identity is a warp. The relation is a symmetry by Lemma \ref{lem:invwrap}. Transitivity follows trivially from the definition of warp equivalency.

Given the close connection between semigroups and semiheaps, see Example \ref{exm:sg-to-sh} and Theorem \ref{monoidsemiheap}, we are compelled to consider a warp analogue for a semigroup.

\begin{definition}\label{def:switch}Let $(S,\cdot)$ be a semigroup. We say that a function $\varphi:S\to S$ is a  {\bf switch} if for all $a,b,c\in G$,
$$
\varphi(a\cdot \varphi(b)\cdot c)=\varphi(c)\cdot b\cdot\varphi(a).
$$
A semigroup equipped with a switch $(S,\cdot, \varphi)$ is called a \textbf{switch semigroup}.
\end{definition}

\begin{example}\label{exm:involutedsg}
Any involuted semigroup $(S,\cdot,*)$, in particular, any group taking $*=(-)^{-1}$, is a switch semigroup since
$$ (a\cdot b^*\cdot c)^*=c^* \cdot b \cdot a^*.$$
\end{example}

\noindent We thus see that switch semigroups generalize involuted semigroups by relaxing the algebraic condition on the map $\varphi$ from an antihomomorphic involution to the switch condition as defined above. The following lemma shows that the switch condition is, nevertheless, sufficient for a switch semigroup to induce a semiheap structure on the base set.

\begin{lemma}\label{lem:sg-to-sh}
Let $(S,\cdot)$ be a semigroup and $\varphi:S\to S$ be a switch. Then $(S,[-,-,-]^\varphi)$, where $[a,b,c]^\varphi=a\cdot \varphi(b)\cdot c$ for all $a,b,c\in S$, is a semiheap.
\end{lemma}
\begin{proof}
Let us check that associativity of semiheaps holds. Let $a,b,c,d,e\in S$,

$$
[[a,b,c]^\varphi,d,e]^\varphi=a\cdot\varphi(b)\cdot c\cdot\varphi(d)\cdot e=[a,b,[c,d,e]^\varphi]^\varphi
$$
and
$$\begin{aligned}
[[a,b,c]^\varphi,d,e]^\varphi &=a\cdot\varphi(b)\cdot c\cdot\varphi(d)\cdot e=a\cdot\varphi(d\cdot\varphi(c)\cdot b)\cdot e=a\cdot \varphi([d,c,b]^\varphi)\cdot e\\ &
=[a,[d,c,b]^\varphi,e]^\varphi.
\end{aligned}
$$
Thus $(S,[-,-,-]^\varphi)$ is a semiheap.
\end{proof}

Moreover, we can state the analogue of Lemma \ref{lem:injsh} for semigroups.

\begin{lemma}\label{lem:equiv-sg-sh}
Let $(S,\cdot)$ be a semigroup and $\varphi:S\to S$ be a function. If there exist $l,r\in S$ such that functions $s\xmapsto{l\cdot}l\cdot s$ and $s\xmapsto{\cdot r}s\cdot r$ are injective, then $(S,[-,-,-]^\varphi)$, where $[a,b,c]^\varphi=a\cdot \varphi(b)\cdot c$ for all $a,b,c\in S$, is a semiheap if and only if $\varphi$ is a switch.
\end{lemma}

\begin{proof}
 If $\varphi$ is a switch, then $(S,[-,-,-]^\varphi)$ is a semiheap by Lemma \ref{lem:sg-to-sh}.
 
 In the opposite direction, let us assume $(S,[-,-,-]^\varphi)$ is a semiheap. Then by associativity of semiheap we get that for all $a,b,c\in S$,
 
$$
l\cdot \varphi(c\cdot\varphi(b)\cdot a)\cdot r=[l,[c,b,a]^\varphi,r]^\varphi=[[l,a,b]^\varphi,c,r]^\varphi=l\cdot\varphi(a)\cdot b\cdot \varphi(c)\cdot r.
$$
Now, since both maps $l\cdot,\cdot r:S\to S$ are injective, we get that 
$$
\varphi(c\cdot\varphi(b)\cdot a)=\varphi(a)\cdot b\cdot\varphi(c).
$$
Thus $\varphi$ is a switch.
\end{proof}

\begin{example}\label{exm:strong}
Observe that assumptions of the Lemma \ref{lem:equiv-sg-sh} are quite strong. If we consider semigroup $S$ with a constant binary operation $(a,b)\xmapsto{\cdot}c$, for a fixed c, then any function $f:S\to S$ is a switch. Even though multiplications by any elements are not injective, we still get a semiheap $(S,[-,-,-]^f).$
\end{example}

\begin{example}\label{exm:ab:raw}
If the semigroup $S$ is abelian, then an identity is a switch. That is $(S,[-,-,-]^{\mathrm{Id}})$ is a semiheap.
\end{example}

\begin{lemma}\label{lem:imp:switch}
Let $(S,[-,-,-])$ be a semiheap and $e,a\in S$ such that $[e,a,x]=[a,e,x]=[x,e,a]=[x,a,e]$, for all $x\in S$. Then $\mu_{ee}$ is a switch for the semigroup $(S,\MiddleB{a})$.
\end{lemma}
\begin{proof}
Let $x,y,z\in S$, then
$$
\begin{aligned}
\mu_{ee}(x\MiddleB{a}\mu_{ee}(y)\MiddleB{a}z)&=[e,[[x,a,[e,y,e]],a,z],e]=[[e,z,a],[x,a,[e,y,e]],e]\\ &=[[e,z,a],[e,y,e],[a,x,e]] =[[e,z,[a,e,y]],e,[a,x,e]]\\ &=[[e,z,[e,a,y]],e,[a,x,e]]=[[e,z,e],[e,y,a],[a,x,e]]\\ &=[[e,z,e],[a,y,a],[e,x,e]]=\mu_{ee}(z)\MiddleB{a}y\MiddleB{a}\mu_{ee}(x),
\end{aligned}
$$
where the second, third and fourth equalities follow by the associativity of a semiheap, the fifth uses the lemma assumption, and the seventh follows in a similar way using associativity and the assumption.
Thus, $\mu_{ee}$ is a switch for the semigroup $(S,\MiddleB{a})$.
\end{proof}

In an abelian case switches and warps coincide, for a particular choice of ternary operation.

\begin{lemma}\label{lem:cor-switch-warp}
Let $(S,\cdot)$ be an abelian semigroup.Then $\varphi:S\to S$ is a switch if and only if $\varphi$ is a warp of $(S,[-,-,-]^{\mathrm{Id}})$.
\end{lemma}

\begin{proof}
If $\varphi$ is a switch and $[a,b,c]^{\mathrm{Id}}=a\cdot b\cdot c$, then
$$
\varphi([c,\varphi(b),a])=\varphi(c\cdot \varphi(b)\cdot a)=\varphi(c)\cdot b\cdot\varphi(a)=[\varphi(c),b,\varphi(a)],
$$
and $\varphi$ is a warp.

If $\varphi$ is a warp for $(S,[-,-,-]^{\mathrm{Id}})$, then
$$
\varphi(a\cdot \varphi(b)\cdot c)=\varphi([a,\varphi(b),c])=\varphi([c,\varphi(b),a])=[\varphi(c),b,\varphi(a)]=\varphi(c)\cdot b\cdot\varphi(a),
$$
where the second equality follows from the fact that $S$ is an abelian semigroup. Thus $\varphi$ is a switch
\end{proof}

\begin{lemma}
Let $(S,[-,-,-])$ be an abelian semiheap and $\varphi$ be a warp such that $\varphi$ is an involution, that is $\varphi^2=\mathrm{Id}$, and there exists $e\in S$ such that $\varphi(e)=e$. Then $\varphi$ is a switch for the semigroup $(S,\MiddleB{e})$.
\end{lemma}
\begin{proof}
Let $a,b,c\in S$, then
$$
\begin{aligned}
\varphi(a\MiddleB{e}\varphi(b)\MiddleB{e}c)&=[[a,e,\varphi(b)],e,c]=\varphi([a,[\varphi(e),\varphi(b),\varphi(e)],c])=\varphi([a,\varphi([e,\varphi^2(b),e]),c])\\ &=[\varphi(a),[e,\varphi^2(b),e],\varphi(c)]=[[\varphi(c),e,b],e,\varphi(a)]=\varphi(c)\MiddleB{e} b\MiddleB{e}\varphi(a),
\end{aligned}
$$
where the second equality follows by the fact that $\varphi(e)=e$ and associativity of a semiheap, third and fourth uses warp property, and fifth follows by the involutivity of $\varphi$. Thus, $\varphi$ is a switch.
\end{proof}

\begin{proposition}\label{prop:sg:list}
Let $(S,\cdot,1)$ be a monoid, $\varphi:S\to S$ a switch and $u\in S$ be such that $\varphi(u)=1$. Then the following statements hold: 
\begin{enumerate}
\item\label{lem:-sg:en:1} A switch $\varphi$ is an involutive antihomomorphism of monoids if and only if $u=1$.
\item \label{lem:-sg:en:2} The element $u\in S$ is invertible with the inverse $\varphi(1)$.
\item\label{lem:-sg:en:3} For all $a,b\in S$, $\varphi(a\cdot b)=\varphi(b)\cdot u\cdot \varphi(a)$ and $\varphi(a)\cdot \varphi(b)=\varphi(b\cdot u^{-1}\cdot a)$.
\item\label{lem:-sg:en:4} For all $a \in S$,  $$\varphi(a)\cdot u^{-1}=\varphi(u^{-1}\cdot a), \qquad u^{-1}\cdot \varphi(a)=\varphi(a\cdot u^{-1}),$$
$$
\varphi(u\cdot a)=\varphi(a)\cdot u, \qquad \varphi(a\cdot u)=u\cdot\varphi(a).$$
\item\label{lem:-sg:en:5} For all $n\in\mathbb{Z}$, $\varphi(u^n)=u^{n-1}.$
\item \label{lem:-sg:en:6} $\varphi$ is a bijection.
\end{enumerate}
\end{proposition}
\begin{proof}
\ref{lem:-sg:en:1}. Assume $\varphi(1)=1$, then for all $a,b\in S$,
$$\varphi(a\cdot b)=\varphi(a\cdot \varphi(1)\cdot b)=\varphi(a)\cdot 1\cdot \varphi(b)=\varphi(a)\cdot\varphi(b)$$
and 
$$
\varphi(\varphi(b))=\varphi(1\cdot\varphi(b)\cdot 1)=\varphi(1)\cdot b\cdot \varphi(1)=b.
$$
Thus, $\varphi$ is an involutive antihomomorphism. Opposite follows by the fact that antihomomorphism of monoids preserve an identity.

\ref{lem:-sg:en:2}. Observe that $\varphi(1)=u^{-1}$, since
$$
u\cdot \varphi(1)=\varphi(u)\cdot u\cdot \varphi(1)=\varphi(u\cdot \varphi(u)\cdot 1)=\varphi(u)=1.
$$
Analogously from the right side, so $\varphi(1)=u^{-1}$ and $u$ is invertible.

\ref{lem:-sg:en:3}. Let $a,b\in S$, then
$$
\varphi(a\cdot b)=\varphi(a\cdot 1\cdot b)=\varphi(a\cdot \varphi(u)\cdot b)=\varphi(b)\cdot u\cdot \varphi(a),
$$
$$
\varphi(a)\cdot \varphi(b)=\varphi(a)\cdot 1\cdot \varphi(b)=\varphi(b\cdot \varphi(1)\cdot a)=\varphi(b\cdot u^{-1}\cdot a).
$$

\ref{lem:-sg:en:4}. Observe that by taking $b=1$ in \ref{lem:-sg:en:3}, we get that
$$
\varphi(a)\cdot u^{-1}=\varphi(a)\cdot \varphi(1)=\varphi(u^{-1}\cdot a)\ \&\ u^{-1}\cdot \varphi(a)=\varphi(a)\cdot \varphi(1)=\varphi(a\cdot u^{-1}),
$$
Similarly, by taking $b=u$, we get
$$
\varphi(a\cdot u)=\varphi(u)\cdot u\cdot\varphi(a)=u\cdot \varphi(a)\ \&\  \varphi(u\cdot a)=\varphi(a)\cdot u\cdot\varphi(u)=\varphi(a)\cdot u,
$$
Thus \ref{lem:-sg:en:4} holds.

\ref{lem:-sg:en:5}. Simply, by taking $a=u^n$ for some $n\in \mathbb{Z}$ in \ref{lem:-sg:en:4}, we get that 
$$
\varphi(u^n)=\varphi(\underbrace{u^{sgn(n)}\cdot\ldots\cdot u^{sgn(n)}}_{|n|-times}\cdot 1)=\varphi(1)\cdot u^{sgn(n)|n|}=u^{-1}\cdot u^{n}=u^{n-1},
$$
where $sgn(n)$ is a sign of a number $n$ and $|n|$ is absolute value of $n$, that is $sgn(n)|n|=n$.

\ref{lem:-sg:en:6}. Let us consider a function $a\xmapsto{\iota_u} u\cdot x\cdot u$. Then the conposition $\iota_u\circ \varphi$ is the inverse function of $\varphi$. Indeed, for all $a\in S$,
$$
\iota_u\circ \varphi\circ \varphi(a)=\iota_u(\varphi(1\circ \varphi(a)\circ 1))=\iota_u(\varphi(1)\cdot a\cdot \varphi(1))=
u\cdot \varphi(1)\cdot a\cdot \varphi(1)\cdot u=u\cdot u^{-1}\cdot a\cdot u^{-1}\cdot u=a,
$$
where second equality follows by the fact that $\varphi$ is a switch and fourth follows by \eqref{lem:-sg:en:2}.
\end{proof}

\begin{example}
Let $(S,\circ,1)$ be an abelian monoid and $x\in S$ be an invertible element. Then $a\xmapsto{\varphi}x^{-1}\circ a$ gives a semiheap structure on $S$ with bracket given for all $a,b,c\in S$ by $[a,b,c]^{x^{-1}\circ}=a\circ x^{-1}\circ b\circ c$.
\end{example}
\begin{example}
Let us consider natural numbers $\mathbb{N}_0:=\mathbb{N}\setminus\{0\}$. Then for any $k\in \mathbb{N}_0$ the map $f_k(m):=k+m$, for all $m\in \mathbb{N}_0$, is a switch.
\end{example}

\section{Biunit pairs}\label{sec:main}
The main goal of this section, and paper, is to state and prove the generalized version of \cite[Theorem 2.8]{wagner1953theory} for biunit pairs, see Theorem \ref{thm:main}. Let us start with the definiton of biunit pairs.

\begin{definition}\label{def:bp}[Biunit Pair cf. \cite{wagner1953theory,wagner1965trans,hollings2017wagner}] Let $(S,[-,-,-])$ be a semiheap and $a,b\in S$. We say that the pair $(a,b)$ is a \textbf{left (right) biunit pair} if for all $x\in S$
\begin{equation}\label{biunit}
    [a,b,x]=x \qquad([x,a,b]=x).
\end{equation}
If a pair is left and right biunit, we call it a \textbf{biunit pair}. 
\end{definition}

Note that for an element $a\in S$ such that $(a,a)$ is a biunit pair we recover the definition of Mal'cev or biunit element.

\begin{example}\label{exm:biunitpairsodd}
Let us consider a set of odd integers  $\mathrm{Odd}:=\{2n+1\ |\ n\in \mathbb{Z}\}$ together with a bracket $[-,-,-]:\mathrm{Odd}^3\to \mathrm{Odd}$ given for all $a,b,c\in \mathrm{Odd}$ by $[a,b,c]=a+b+c$. Obviously odd numbers are closed on the bracket operation. Moreover, an element is a biunit if and only if $a+a+b=b$, which implies $a=0$,however $0\not\in \mathrm{Odd}$. Now, observe that a pair $(a,-a)$ is a biunit pair for all $a\in \mathrm{Odd}$.
\end{example}

\begin{example}\label{exm:biunitpairsfish}
In the algebra of cubic matrices \cite{zapata2022heaps}, the ternary operation on $3$-index arrays of scalars (elements of a semiring in general) of size $N\times N \times N$ defined by
\begin{equation*}
    [a,b,c]_{ijk}:=\sum_{p,q,r=1}^N a_{ijp}\cdot b_{qrp} \cdot c_{qrk}
\end{equation*}
is easily checked to be a semiheap operation. The pair of special cubic matrices generalizing the identity matrix defined by
\begin{equation*}
I_{ijk}=\delta_{ijk} \qquad \iota_{ijk}=\delta_{ik},    
\end{equation*}
form a (right) biunit pair since we can directly compute:
\begin{equation*}
    [a,I,\iota]_{ijk}=a_{ijk}=[a,\iota,I]_{ijk}.
\end{equation*}
\end{example}

\begin{lemma}\label{lem:biunit:prop}
Let $(S,[-,-,-])$ be a semiheap. Then
\begin{enumerate}
    \item\label{bi:eq1} If $(a,b)$ is a left (right) biunit pair then $(b,a)$ is a left (right) biunit pair. 
    \item\label{bi:eq2}  For any $a\in S$, if there exists $b\in S$ such that $(a,b)$ is a biunit pair, then $b$ is unique.
     \item\label{bi:eq3}  If $(a,b)$ is a left (right) biunit pair. Then if $b$ is a right (left) biunit, then $a$ is a left (right)  biunit.
 
\end{enumerate}
\end{lemma}
\begin{proof}
\ref{bi:eq1}. Observe that if $(a,b)\in S\times S$ is a left biunit pair then $[a,b,b]=b$, and for all $x\in S$,
$$
[b,a,x]=[[a,b,b],a,x]=[a,[a,b,b],x]=[a,b,x]=x,
$$
thus $(b,a)$ is a left biunit pair. Similarly for the right case.

\ref{bi:eq2}. Observe that if $(a,b)$ and $(a,b')$ are biunit pairs, then
$$
b'=[[a,b,b'],a,b]=[a,[a,b',b],b]=[a,b,b]=b,
$$
and $b$ is unique.

\ref{bi:eq3}. Let us assume that $(a,b)$ is a biunit pair and $b$ is a right biunit. Then for all $x\in S$

$$
x=[a,b,[b,a,x]]=[a,[a,b,b],x]=[a,a,x],
$$
so $a$ is a left biunit. Proof for the right biunit pair is analogous.
\end{proof}


\begin{lemma}\label{lem:biunit-pre}
Let $S$ be a semiheap with a biunit pair $(a,b)$, then $S_{\varphi}$, for a bijective warp $\varphi$, also has a biunit pair, explicitly $(a,c)$, where $\varphi(c)=b$.
\end{lemma}
\begin{proof}
Let $(a,b)$ be a biunit pair and $\varphi$ be a bijective warp. Then there exist $c\in S$ such that $\varphi(c)=b$, and for all $x\in S$
$$[a,c,x]_\varphi=[a,\varphi(c),x]=[a,b,x]=x,$$
and 
$$\varphi([x,\varphi(a),c])=[\varphi(x),a,\varphi(c)]=\varphi(x).$$
Observe that since $\varphi$ is injective we get that
$$
[x,a,c]_{\varphi}=[x,\varphi(a),c]=x,
$$
and $(a,c)$ is a biunit in $S_\varphi$.
\end{proof}
\begin{example}
Observe that not all semiheap structures on the same base set are warp equivalent. Let us denote by $\mathrm{H}(\mathbb{Z})$ a heap $(\mathbb{Z},[-,-,-])$, where $[a,b,c]=a-b+c$ for all $a,b,c\in \mathbb{Z}$, and by $\mathrm{H}(\mathbb{Z}_0)$ a semiheap, where $[a,b,c]=0$ for all $a,b,c\in \mathbb{Z}$. Semiheaps $\mathrm{H}(\mathbb{Z})$ and $\mathrm{H}(\mathbb{Z}_0)$ are not warp equivalent as by Lemma \ref{lem:biunit-pre}, warp equivalence preserve biunit pairs.
\end{example}

\begin{lemma}
Let $S$ be a monoid with two different biunit pairs $(a,b)$ and $(c,d)$. Then monoids $(S,\MiddleB{a},b)$ and $(S,\MiddleB{c},d)$ are isomorphic. Moreover, switches $\mu_{bb}$ and $\mu_{dd}$ are related by equation
$$
\lambda_{da}\circ \mu_{bb}=\mu_{dd}\circ \lambda_{ad}^{-1}
$$
\end{lemma}

\begin{lemma}\label{lem:bicurryings}
Let $(S,[-,-,-])$ be a semiheap with a biunit pair $(a,b)$. Then the set 
$$U:=\langle \mathrm{Id}, \mu_{aa}, \mu_{bb}, \mu_{ab}, \mu_{ba} \rangle,$$
where $\mu$ are the middle shifts induced by the pair $(a,b)$, see \eqref{def:unaries}, is a subgroup of the group $(\mathrm{Aut}_{\mathrm{Set}}(S),\circ)$, that is a group of bijections with the composition of functions.
\end{lemma}

\begin{proof}
Since all of the curryings of a pair of elements are endofunctions, it is enough to observe that all the generators are invertible. One can easily check that 
$$
\begin{aligned}
\mu_{aa} \circ \mu_{bb} &= \mu_{ab}^2=\mu_{ba}^2 = \lambda_{ab} = \rho_{ab} = \mathrm{Id},
\end{aligned}
$$
and $(U,\circ)$ is a group, thus it is a subgroup of $(\mathrm{Aut}_{\mathrm{Set}}(S),\circ)$. 
\end{proof}

\begin{corollary}\label{cor:warpequiv}
Observe that all the unary curryings induced by a biunit pair $(a,b)$ belongs to the set $U$. Moreover,  $S$ is wrap equivalent to $S_{\lambda_{aa}}$, $S_{\lambda_{bb}}$, $S_{\rho_{aa}}$ and $S_{\rho_{bb}}$.
\end{corollary}
\begin{proof}
Observe that if $(a,b)$ is a biunit pair, then
$$
\begin{aligned}
    &\mu_{aa} \circ \mu_{ab} = \mu_{ba} \circ \mu_{aa} = \rho_{aa},\  \mu_{bb} \circ \mu_{ba} = \mu_{ab} \circ \mu_{bb} = \rho_{bb},\\
    &\mu_{aa} \circ \mu_{ab} = \mu_{ba} \circ \mu_{aa} = \lambda_{aa},\  \mu_{bb} \circ \mu_{ba} = \mu_{ab} \circ \mu_{bb} = \lambda_{bb},
\end{aligned}
$$
thus $\rho_{aa},\rho_{bb},\lambda_{aa},\lambda_{bb}\in U$. Since, by Lemma \ref{lem:shifts}, all of those left and right shifts are warps and are invertible with inverses:
$$
\lambda_{aa} \circ \lambda_{bb} = \rho_{aa} \circ \rho_{bb} = \mathrm{Id},
$$
we get that those are warp equivalences.
\end{proof}

\begin{lemma}
Let $(S,[-,-,-])$ be a semiheap with a biunit pair $(a,c)$. Then curryings $\mu_{ac}$ and $\mu_{ca}$ are anti-automorphism of semigroups $(S,\MiddleB{a})$ and $(S,\MiddleB{c})$. Moreover, $\mu_{aa}:(S,\MiddleB{a})\to (S,\MiddleB{c})$ is an anti-isomorphism of semigroups.
\end{lemma}

\begin{proof}
Indeed, let $x,y\in S$, then
$$
\begin{aligned}
\mu_{ac}(x)\MiddleB{a}\mu_{ac}(y)&=[[a,x,c],a,[a,y,c]]=[a,[a,c,x],[a,y,c]]=[a,x,[a,y,c]]\\ &=[a,[y,a,x],c]=\mu_{ac}(y\MiddleB{a}x),
\end{aligned}$$
where third and fifth follows by the associativity of semiheap and the fourth one follows by the fact that $(a,c)$ is a biunit pair. Thus $\mu_{ac}$ is an anti-endomorphism of the semigroup $(S,\MiddleB{a})$. Since $\mu_{ac}$ is an involution, due to Lemma \ref{lem:bicurryings}, $\mu_{ac}$ is an automorphism. Proof for the semigroup $(S,\MiddleB{c})$ is analogous. The case of  $\mu_{ca}$ can be proven the same way. 

The second statement is observation that for all $x,y\in S$,
$$
\begin{aligned}
\mu_{aa}(x\MiddleB{a}y)&=[a,[x,a,y],a]=[a,[x,[a,c,a],y],a]=[[a,y,[a,c,a]],x,a]=[[a,y,a],c,[a,x,a]]\\ &=\mu_{aa}(y)\MiddleB{c}\mu_{aa}(x).
\end{aligned}
$$
Thus, $\mu_{aa}$ is an anti-homomorphism, and by Lemma \ref{lem:bicurryings}, it is an anti-isomorphism.
\end{proof}

\begin{lemma}\label{lem:biunit:subset}
A set of all left (right) biunit pairs is a subsemiheap of a product semiheap $S\times S$ with ternary operation given on coordinates, that is for all left biunit pairs $(a,b),(c,d),(e,g)$ $[(a,b),(c,d),(e,g)]:=([a,c,e],[b,d,g])$.
\end{lemma}
\begin{proof}
Let $(a,b),(c,d)$ and $(e,g)$ be biunit pairs and $x\in S$, then 
$$
\begin{aligned}
[[a,c,e],[b,d,g],x]&=[[[a,c,e],g,d],b,x]=[[a,c,[e,g,d]],b,x]=[[a,c,d],b,x]=[a,[d,c,b],x]\\ &=[a,b,x]=x.
\end{aligned}
$$
Thus $([a,c,e],[b,d,g])$ is a left biunit pair. Associativity, follows from the definition of a ternary operation on the product.
\end{proof}

\begin{corollary}
A set $\mathrm{B}:=\{a\in S\ |\ \exists{b\in S} \ (a,b) \text{\  is\  a\  biunit\  pair}\}$ is a subsemiheap of $S$.
\end{corollary}

\begin{lemma}\label{lem:mainproof1}
Let $(S,[-,-,-])$ be a semiheap with a biunit pair $(a,b)$. Then $\mu_{bb}$ is a switch for the semigroup $(S,\MiddleB{a})$ and for all $x,y,z\in S$ $[x,y,z]=x\MiddleB{a}\mu_{bb}(y)\MiddleB{a}z$.
\end{lemma}

\begin{proof}
A biunit pair $(a,b)$ satisfy conditions in Lemma \ref{lem:imp:switch}, and therefore $\mu_{bb}$ is a switch for $(S,\MiddleB{a})$. Second claim is a simple observation that
$$x\MiddleB{a}\mu_{bb}(y)\MiddleB{a}z=[[x,a,[b,y,b]],a,z]=[[x,a,b],y,[b,a,z]]=[x,y,z].$$
Q.E.D.
\end{proof}

\begin{corollary}\label{cor:mainproof2}
If $(S,[-,-,-])$ is a semiheap with a biunit pair $(a,b)$, then $(S,\MiddleB{a})$ is a monoid with an identity $b$ and $\mu_{bb}(b)\not=b$ in general, but $\mu_{bb}(a)=b$. Thus, $\mu_{bb}$ is a bijection and $a$ is invertible by Proposition \ref{prop:sg:list} \eqref{lem:-sg:en:2} and \eqref{lem:-sg:en:6} with inverse $[b,b,b]$. If $\mu_{bb}(b)=b$, we get that $b$ is the inverse of $a$, and $a$ is a biunit.
\end{corollary}

\begin{proposition}
Let $S$ be a semiheap with two different biunit pairs $(a,b)$ and $(c,d)$. Then monoids $(S,\MiddleB{a},b)$ and $(S,\MiddleB{c},d)$ are isomorphic. Moreover, switches $\mu_{bb}$ and $\mu_{dd}$ are related by equation
$$
\lambda_{da}\circ \mu_{bb}=\mu_{dd}\circ \lambda_{ad}^{-1}.
$$
\end{proposition}

\begin{proof}
Let us check that $\lambda_{da}$ is an isomorphism of monoids. Indeed for all $x,y\in S$,
$$
\begin{aligned}
\lambda_{da}(x)\MiddleB{c}\lambda_{da}(y)&=[[d,a,x],c,[d,a,y]]=[d,a,x],[a,d,c],y]=[[d,a,x],c,y]\\ &=[d,a,[x,c,y]]=\lambda_{da}(x\MiddleB{c}y),\\
\lambda_{da}(b)&=[d,a,b]=d,\\
\lambda_{da}\circ \lambda_{bc}(x)&=[d,a,[b,c,x]]=x.
\end{aligned}
$$
For the second statement, observe that $\lambda_{cb}=\lambda_{ad}^{-1}$ as for all $x\in S$,
$$\lambda_{cb}\circ\lambda_{ad}(x)=[c,b,[a,d,x]]=[[c,b,a],d,x]=[c,d,x]=x.$$
Now, 
$$
\lambda_{da}\circ \mu_{bb}(x)=[d,a,[b,x,b]]=[d,x,b]=[d,x,[d,c,b]]=\mu_{dd}\circ \lambda_{cb}=\mu_{dd}\circ \lambda_{ad}^{-1}.
$$
Q.E.D.

\end{proof}

We are now ready to prove our main theorem extending the classic result by Wagner \cite[Theorem 2.8]{wagner1953theory} stated in Theorem \ref{monoidsemiheap}.

\begin{theorem}\label{thm:main}
Let $(S,\cdot,1)$ be a monoid and $\varphi$ be a bijective switch. Then a pair $(\varphi^{-1}(1),1)$ is a biunit pair in the semiheap $(S,[-,-,-]^{\varphi})$. Furthermore, every semiheap $S$ with a biunit pair $(a,b)$ is a monoid $(S,\MiddleB{a},b)$ with a bijective switch $\mu_{bb}$.
\end{theorem}
\begin{proof}
The second part is Lemma \ref{lem:mainproof1} and Corollary \ref{cor:mainproof2}. For the first part, $(S,[-,-,-]^\varphi)$ is a semiheap by Lemma \ref{lem:sg-to-sh}, so it is enough to check that $(\varphi^{-1}(1),1)$ is a biunit pair. Indeed, observe that for all $c\in S$,
$$
[\varphi^{-1}(1),1,c]^\varphi=\varphi^{-1}(1)\cdot\varphi(1)\cdot c=c=c\cdot \varphi(1)\cdot \varphi^{-1}(1)=[c,1,\varphi^{-1}(1)]^\varphi,
$$
where $\varphi(1)$ is the inverse of $\varphi^{-1}(1)$ by the Proposition \ref{prop:sg:list} \eqref{lem:-sg:en:2}.
By Lemma \ref{lem:biunit:prop} \eqref{bi:eq1} it is enough to check just this one equality.
\end{proof}

\begin{corollary}\label{correspondence}
There is a one-to-one correspondence between semiheaps with a chosen biunit pair and monoids with a bijective switch:
\begin{equation*}
    \{ \text{semiheap} + \text{biunit} \,\, \text{pair} \} \rightleftharpoons \{\text{monoid} +  \text{bijective} \,\, \text{switch}\}
\end{equation*}
which is realized explicitly via the assignments:
\begin{equation*}
(S,(a,b))\xmapsto{\Omega} (S,\MiddleB{a},b,\mu_{bb})\quad \text{and} \quad (M,\cdot,1,\varphi)\xmapsto{\Lambda} ((M,[-,-,-]^{\varphi}), (\varphi^{-1}(1),1)).
\end{equation*}
\end{corollary}
\begin{proof}
It is enough to show that $\Lambda=\Omega^{-1}$. Indeed, let us consider 
$$\Lambda\circ \Omega (S,(a,b))=\Lambda(S,\MiddleB{a},b,\mu_{bb}).$$
Observe that for all $x,y,z\in S$, $$[x,y,z]^{\mu_{bb}}=x\MiddleB{a}\mu_{bb}(y)\MiddleB{a}z=[[x,a,[b,y,b]],a,z]=[x,y,z]$$
and 
$$\mu_{bb}^{-1}(b)=\mu_{aa}(b)=a,$$
where $[-,-,-]$ is the ternary operation of $S$. Thus $\Lambda(S,\MiddleB{a},b,\mu_{bb})=(S,(a,b))$, and $\Omega$ is a right inverse of $\Lambda$.

For the left inverse, we have  that $\Omega(\Lambda(M,\cdot,1,\varphi))=\Omega(M,[-,-,-]^{\varphi},(\varphi^{-1}(1),1))$. Let us denote $e:=\varphi^{-1}(1)$, observe that for all $x,y\in M$,
$$
a\MiddleB{e}b=a\cdot \varphi(e)\cdot b=a\cdot \varphi(\varphi^{-1}(1))\cdot b=a\cdot b,
$$
$1$ is the identity from the definition of $\Omega$, and 
$$
\mu_{11}(x)=1\cdot \varphi(x)\cdot 1=\varphi(x).
$$
Thus $\Omega((M,[-,-,-]^{\varphi}),(\varphi^{-1}(1),1))=(M,\cdot,1,\varphi)$, and $\Omega$ is a right inverse to $\Lambda$. Therefore $\Lambda$ is a bijection and we have one-to-one correspondence.
\end{proof}

These results motivate the definition of the family of semiheaps whose elements belong to biunit pairs.

\begin{definition}\label{diheap}
A semiheap $(D,[-,-,-])$ in which for every element $a\in D$ there exists $b\in D$ such that $(a,b)$ is a biunit pair, is called a \textbf{diheap}.
\end{definition}

\begin{example}\label{ex:abeliandiheap}
Any abelian group $(A,+)$ has an associated abelian semiheap structure $(A,[-,-,-])$ given by $[a,b,c]:=a+b+c$. The inverse map ensures that all elements $a\in A$ are in a pair $(a,-a)$ satisfying:
$$
[a,-a,x]=a-a+x=x=x+a-a=[x,a,-a]
$$
and thus $(A,[-,-,-])$ is a diheap, that otherwise fails to be a heap in general.
\end{example}

\noindent It follows from this definition that diheaps generalize heaps. In fact, there is a deep connection between diheaps and heaps via twisting.

\begin{proposition}\label{prop:equiv}
A diheap $(D,[-,-,-])$ is warp equivalent to a heap $D_\psi$, where $\psi$ is an involutive semiheap automorphism, and thus a warp.
\end{proposition}
\begin{proof}
Let $a,b,c,d\in D$ and let us denote by $\hat{a}$, a second component of a biunit pair, that is an element of $D$ such that $(a,\hat{a})$ is a biunit pair. Then, from Lemma \ref{lem:biunit:prop} we get that the map $\psi(a):=\hat{a}$ is a well-defined function and an involution. The only thing to check is that $\psi$ is a warp. Observe that 
$$
[[a,b,c],[\hat{a},\hat{b},\hat{c}],d]=[[a,b,[c,\hat{c},\hat{b}]],\hat{a},d]=d.
$$
Thus $\widehat{[a,b,c]}=[\hat{a},\hat{b},\hat{c}]$ and
$$
\psi([a,\psi(b),c])=\widehat{[a,\hat{b},c]}=[\hat{a},\hat{\hat{b}},\hat{c}]=[\hat{a},b,\hat{c}]=[\psi(a),b,\psi(c)].
$$
Therefore $\psi$ is a warp, and one can easily check that every element in $D_\psi$ is a biunit. Thus, $D$ is warp equivalent to a heap.
\end{proof}

\noindent Note that if the diheap $(D,[-,-,-])$ in Proposition \ref{prop:equiv} above is a heap, i.e. all biunit pairs are of the form $(a,a)$ for $a\in D$, then the induced involutive semiheap automorphism is simply the identity $\psi=\text{id}_D$. This suggest that heaps and diheaps are distinct classes of semiheaps, as we can indeed confirm.

\begin{proposition}
Considered as semiheaps, a diheap is isomorphic to a heap if and only if it is, itself, a heap.
\end{proposition}
\begin{proof}
Firstly, observe that any surjective homomorphism $\varphi:D\to S$, for some semiheaps $D$ and $S$, preserve biunit pairs. Indeed, if $(a,b)$ is a biunit pair in $D$, then for all $y\in S$ exists $x\in D$ such that $\varphi(x)=y$ and 
$$
[\varphi(a),\varphi(b),y]=\varphi([a,b,x])=\varphi(x)=y=\varphi(x)=\varphi([x,a,b])=[y,\varphi(a),\varphi(b)].
$$
Thus $(\varphi(a),\varphi(b))$ is a biunit pair in $S$. Now, if $D$ is a diheap which is not a heap, i.e. there exists at least one biunit pair $(a,b)$ such that $a\neq b$, and $S$ is a heap, it follows that $\varphi(a)=\varphi(b)$ thus $\varphi$ fails to be injective. Conversely, if diheap is a heap, it is isomorphic to itself.
\end{proof}

Diheaps thus appear as a well-motivated novel class of semiheaps that strictly generalizes heaps. This poses some intriguing lines of further enquiry. Although we have shown that diheaps are generally non-isomorphic to heaps, one can ask whether there exists a different category of semiheaps where warp equivalence corresponds to isomorphism. It seems plausible that the diheaps induced from invertible structures as in Example \ref{ex:abeliandiheap} are essentially the only class of non-heap diheaps. Following Theorem \ref{thm:main}, diheaps correspond to a particular class of switch monoids, it will be interesting to consider whether such class of algebras has been identified before or whether it may have some application in semigroup research.

\section*{Acknowledgements}
\noindent The research of B. Rybo\l owicz is supported by the EPSRC research grant EP/V008129/1. We would like to thank Mark Verus Lawson, Andr\'es Ortiz-Mu\~noz and Jos\'e Figueroa-O'Farrill for all the feedback and helpful conversations.

\printbibliography
\end{document}